\documentclass[12pt]{article}
\usepackage[margin=1in]{geometry}

\usepackage[english]{babel}
\usepackage{graphicx,epstopdf,epsfig}
\usepackage{amsfonts,epsfig,fancyhdr,graphics, hyperref,amsmath,amssymb}
\usepackage{tikz}
\usepackage{amsthm,bbm}

\newtheorem{thm}{Theorem}[section]
\newtheorem{cor}[thm]{Corollary}
\newtheorem{lem}[thm]{Lemma}

\newtheorem{conj}[thm]{Conjecture}
\newtheorem{obs}[thm]{Observation}

\newtheorem{example}[thm]{Example}

\allowdisplaybreaks

\renewcommand{\L}{\mathcal{L}} 
\newcommand{\K}{\mathcal{K}}
\renewcommand{\v}{\mathbf{v}}
\renewcommand{\u}{\mathbf{u}}
\newcommand{\w}{\mathbf{w}}
\newcommand{\x}{\mathbf{x}}
\newcommand{\y}{\mathbf{y}}

\newcommand{\dlambda}[1]{\frac{d \lambda}{d #1}} 
\newcommand{\dk}[1]{\frac{d \K}{d #1}} 
\newcommand{\spec}{\text{Spec}}

\providecommand{\keywords}[1]
{
  \small	
  \textbf{\textit{Keywords--}} #1
}

\providecommand{\AMS}[1]
{
  \small	
  \textbf{\textit{AMS:}} #1
}

\usepackage{bbm}
\usepackage{subcaption}
\usepackage{diagbox}

\tikzstyle{vertex}=[draw,thick,fill=white,circle,inner sep=2pt]
\tikzstyle{full}=[draw,thick,fill=black,circle,inner sep=2pt]
\tikzstyle{empty}=[draw,color=black!40!white,thick,fill=white,circle,inner sep=2pt]
\usetikzlibrary{hobby}

\begin{document}
\bibliographystyle{plain}

\setcounter{page}{1}

\thispagestyle{empty}

\title{\Title\thanks{Received
 by the editors on \DoS.
 Accepted for publication on \DoA. 
 Handling Editor: \HE. Corresponding Author: \CA}}

\title{On the Edge Derivative of the Normalized Laplacian with Applications to Kemeny's Constant}

\author{Connor Albright\thanks{UC Riverside, Riverside, California, USA (cpcalbright2010@gmail.com)} 
\and 
Kimberly P. Hadaway
\thanks{Iowa State University, Ames, Iowa, USA (\{kph3, joeljeff\}@iastate.edu)} 
\and 
Ari Holcombe Pomerance\thanks{Macalester College, Saint Paul, Minnesota, USA (aholcomb@macalester.edu)}
\and
Joel Jeffries\footnotemark[2] 
\and
Kate J. Lorenzen\thanks{Linfield University, McMinniville, Oregon, USA (klorenzen@linfield.edu)} 
\and
Abigail K. Nix\thanks{Middlebury College, Middlebury, Vermont, USA (anix@middlebury.edu)} }

\maketitle

\begin{abstract} 
In a connected graph, Kemeny's constant gives the expected time of a random walk from an arbitrary vertex $x$ to reach a randomly-chosen vertex $y$. 
Because of this, Kemeny's constant can be interpreted as a measure of how well a graph is connected. 
It is generally unknown how the addition or removal of edges affects Kemeny's constant.
Inspired by the edge derivative of the normalized Laplacian, we derive the edge derivative of Kemeny's constant for several graph families. 
In addition, we find sharp bounds for the edge derivative of an eigenvalue of the normalized Laplacian and bounds for the edge derivative of Kemeny's constant. 
\end{abstract}


\AMS{05C50, 05C81.}

\section{Introduction}
A graph $G$ has a set of vertices $V(G)$ and a set of edges $E(G)$. 
An edge connecting vertices $x$ and $y$ is written $e= \{x,y\}$, and we say that $x$ and $y$ are \emph{adjacent} if there exists an edge between them. 
We say that a graph is \emph{simple} if it has no loops (an edge going from one vertex back to itself) and no more than one edge between any two vertices.
For a vertex $x$, the \emph{neighbors} of $x$ are the vertices adjacent to $x$, and the \emph{degree} of $x$, denoted $d_x$, is the number of neighbors of $x$.
A graph is \emph{connected} if for any pair of vertices $x,y$, there exists a path, or a sequence of edges, from $x$ to $y$.

Consider a random walk on the vertices and edges of a simple connected graph $G$. 
We can think of this as someone walking along the edges of the graph. 
Our walker starts by occupying vertex $x$, and in the next step, the walker moves to one of the neighbors of $x$ at random with uniform probability of $\frac{1}{d_x}$. 
Note that this is a finite Markov chain, whose probability transition matrix is defined as $T:=D^{-1}A$. Here, $D$ is the diagonal matrix containing the vertex degrees, and $A$ is the adjacency matrix of $G$. 
Since $G$ is connected, $D$ is invertible. 

When analyzing Markov chains, or random walks on graphs, we can look at the long-term or the short-term behavior. 
The long-term behavior is found by taking repeated powers of $T$. 
For a general Markov chain which is irreducible and primitive, each row converges to the \emph{stationary vector} $\w$ of the chain. 
For more details, see \cite{GS97}. 
This vector $\w$ is also a left eigenvector of $T$ with the corresponding eigenvalue $1$. 
The stationary vector can be interpreted as where a random walker is likely to be during a long random walk, and is independent of the starting vertex. 
The short-term behavior is described by the \emph{mean-first passage times}. 
These indicate the expected time (or number of steps), starting at some vertex $x$, to reach some other vertex $y$, denoted $m_{x,y}$.

Kemeny's constant, denoted $\mathcal{K}(G)$, combines the mean-first passage times and the stationary vector of a graph $G$. 
For a vertex $x$, the weighted average of the mean-first passage times from $x$ to each other vertex in the graph, where the weights are the corresponding entries of the stationary vector results in the following parameter,
\begin{align*}
    \mathcal{K}(G)= \sum_{y=1}^n m_{x,y}\w_y.
\end{align*}
Surprisingly, $\mathcal{K}$ is not dependent on the starting vertex $x$, hence the name Kemeny's constant. 

Kemeny's constant has many useful interpretations, including the spread of infectious diseases (how quickly a disease will reach epidemic levels), molecular conformation dynamics (presence or absence of metastable sets), and urban road networks (how well connected a network is). 
In general, a lower Kemeny's constant means that a graph is more connected, and a higher Kemeny's constant means that a graph is less connected. 

\begin{figure}[ht]
    \centering
   \begin{tikzpicture}
\foreach \x in {1,...,6}{
    \node[vertex] (c\x) at (\x,0) {};
    }
\node[vertex] (a1) at ({0 + cos(60)},{sin(60)}) {};
\node[vertex] (a6) at ({7 + cos(60)},{sin(60)}) {};

\node[vertex] (b1) at ({0 +cos(-60)},{sin(-60)}) {};
\node[vertex] (b6) at ({7 +cos(-60)},{sin(-60)}) {};

\draw[thick] (a1)--(b1)--(c1)--(a1) (a6)--(b6)--(c6)--(a6);
\draw[thick]  (c1)--(c2)--(c3)--(c4)--(c5)--(c6);

\node[vertex] (d1) at ({0 + cos(120)},{sin(120)}) {};
\node[vertex] (d2) at ({0 + cos(180)},{sin(180)}) {};
\node[vertex] (d3) at ({0 + cos(-120)},{sin(-120)}) {};
\foreach \x in {1,...,3}{
    \draw[thick] (d\x)--(a1) (d\x)--(b1) (d\x)--(c1);
    \foreach \y in {1,...,\x} {
        \ifnum\x=\y\relax\else 
            \draw[thick] (d\x)--(d\y);
        \fi
    }
}
\node[vertex] (e1) at ({7 + cos(120)},{sin(120)}) {};
\node[vertex] (e2) at ({7 + cos(0)},{sin(0)}) {};
\node[vertex] (e3) at ({7 + cos(-120)},{sin(-120)}) {};
\foreach \x in {1,...,3}{
    \draw[thick] (e\x)--(a6) (e\x)--(b6) (e\x)--(c6);
    \foreach \y in {1,...,\x} {
        \ifnum\x=\y\relax\else
            \draw[thick] (e\x)--(e\y);
        \fi
    }
}
\end{tikzpicture}
    \caption{A barbell graph consists of two cliques connected by a path. Removing an edge within the clique will decrease Kemeny's constant. This contradicts our intuition that removing edges will always make a graph less connected. }
    \label{fig:barbell}
\end{figure}
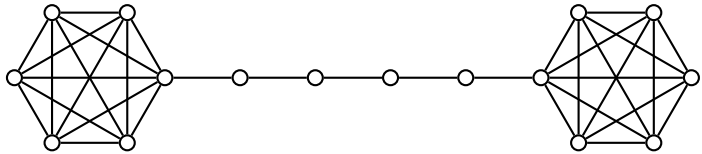

For these and other applications, a big question is: \textit{How do changes in the network lead to changes in Kemeny's constant?}
Given the relation to connectivity, one might assume Kemeny's constant must decrease as edges are added to a graph. However, contrary to this intuition, there are some graphs where the removal of an edge decreases Kemeny's constant, and some graphs where the addition of an edge increases Kemeny's constant (see Figure \ref{fig:barbell}). This contradictory phenomenon is well documented as Braess' paradox where adding roads to a road network can slow down traffic (see \cite{B68}).  
An open problem is how the removal or addition of edges affects Kemeny's constant. 
Breen and Kirkland in \cite{BK19} looked at how small perturbations in the transitional probabilities related to changes in Kemeny's constant. 
They were able to find a condition number which could serve as a confidence interval if the probabilities were calculated with raw data. 
They used the fundamental matrix of the perturbed transition matrix to find the size of the change. 
We also look at small changes to probabilities yet7 remain in the context of simple connected graphs. 

In particular, we use the connection between Kemeny's constant and the spectrum of the 
normalized Laplacian matrix to compute how Kemeny's constant changes as small changes are made to a graph. 
The \emph{normalized Laplacian matrix} of a graph is given by $\L = I - D^{-\frac{1}{2}}AD^{-\frac{1}{2}}$ where $A$ is the adjacency matrix and $D$ is the degree diagonal matrix.
Note that the probability transition matrix $T$ is similar to $D^{-\frac{1}{2}}AD^{-\frac{1}{2}}$. Therefore, building off the connection Levene and Loizou made between Kemeny's constant and the spectrum of $T$ in \cite{LL02},  Kemeny's constant can be computed as
\begin{align*}
    \mathcal{K}(G) &= \sum_{\lambda_{T} \ne 1} \frac{1}{1-\lambda_{T}}\\
    &=\sum_{\lambda_{\L} \ne 0} \frac{1}{\lambda_{\L}},
\end{align*}
where each $\lambda_\L$ is an eigenvalue of $\L$.

Calculating Kemeny's constant using the spectra of graph matrices allows us to use existing tools of spectral graph theory.
Recently, there have been developments in the edge derivative of the eigenvalues of graph matrices by Aksoy, Purvine, and Young in \cite{APY21}. 
They derived the derivative of eigenvalues with respect to a vertex $x$, where all edges containing $x$ had a parameter $t$ added to their current weight. (Here, we assume all unweighted edges have weight $1$ and all non-edges have weight $0$.) 
This idea can be extended to any set of edges (or non-edges) of the graph.
 
This edge derivative can be interpreted as the effect of a slight change to an edge (or non-edge) weight on the eigenvalues of a graph matrix. 
Since Kemeny's constant can be calculated by the eigenvalues of the normalized Laplacian, it follows that the edge derivative can be extended to Kemeny's constant. 
In this paper, we give results on the edge derivative of eigenvalues of the normalized Laplacian and the edge derivative of Kemeny's constant. 
We explicitly find these values for some families of graphs and establish bounds. 

\section{Edge Derivative of Eigenvalues}\label{sec:eigenvalue}

Before we can discuss the derivative of an eigenvalue, we first establish the parameterized normalized Laplacian is, in fact, differentiable. Let us first establish some facts about the spectrum of a matrix (see \cite{lancaster64}). 

\begin{thm}
    Let $M_0$ be a real-symmtric $n \times n$ matrix with eigenpair $(\lambda_0, \v_0)$ such that $\v_0$ is a unit vector. If $\lambda_0$ is a simple eigenvalue, then there exists a neighborhood $N(M_0)$ and functions $\lambda:N(M_0)\to \mathbb{R}$ and $\v:N(M_0)\to \mathbb{R}^n$ such that the eigenpair of $M \in N(M_0)$ is $(\lambda(M), \v(M))$ such that $\v(M)$ is a unit vector. 

    Furthermore, $\lambda$ and $\v$ are infinitely differentiable on $N(M_0)$.
\end{thm}

Aksoy et al. in \cite{APY21} represented $M(t)=A+tB$ where $A$ is the adjacency matrix of a graph $G$ and $B$ is a symmetric matrix representing a collection of edges. In their work, they looked at all edges incident to a particular vertex (so the derivative is in the direction of a vertex). 

We will not restrict B and instead have $B$ be a symmetric matrix representing any collection of edges denoted $E_C$. An example of this parameterization is shown in Figure \ref{fig:pawandL}. 
Colloquially, we will refer to this as the \textit{edge derivative} of a graph.

 \begin{figure}[ht]
\centering
\begin{subfigure}{.35\textwidth}
    \centering
    \begin{tikzpicture}[scale=.9]
     \node[vertex, label=$u_3$] (v1) at ({0 + cos(0)},{sin(0)}) {};
     \node[vertex, label=$u_1$] (v2) at ({0 + cos(120)},{sin(120)}) {};
     \node[vertex, label=below:$u_2$] (v3) at ({0 + cos(240)},{sin(240)}) {};
     \node[vertex, label=$u_4$] (v4) at ({0 + cos(0) + abs(sin(120) - sin(240))}, 0) {};
     
     \draw[thick] (v1)--(v2) node[midway, above] {$1$};
     \draw[thick] (v2)--(v3) node[midway, left] {$1$};
     \draw[thick] (v3)--(v1) node[midway, below] {$1$};
     \draw[ultra thick] (v1)--(v4) node[midway, above] {$1+t$};
    \end{tikzpicture}
\end{subfigure}
\begin{subfigure}{.6\textwidth}
\centering
    \begin{align*}
        \L(t) = \left[ \begin{array}{cccc}
          1  & \frac{-1}{2} &\frac{-1}{\sqrt{2(3+t)}} & 0 \\
          \frac{-1}{2}  & 1 &\frac{-1}{\sqrt{2(3+t)}} & 0\\
          \frac{-1}{\sqrt{2(3+t)}} & \frac{-1}{\sqrt{2(3+t)}} & 1 & \frac{-1-t}{\sqrt{(1+t)(3+t)}}\\
          0 & 0 & \frac{-1-t}{\sqrt{(1+t)(3+t)}} & 1
        \end{array}\right]
    \end{align*}
\end{subfigure}
\caption{The parameterized paw graph with $E_C = \{\{u_3,u_4\}\}$ (i.e. one changing edge), and the associated normalized Laplacian.}
\label{fig:pawandL}
\end{figure}
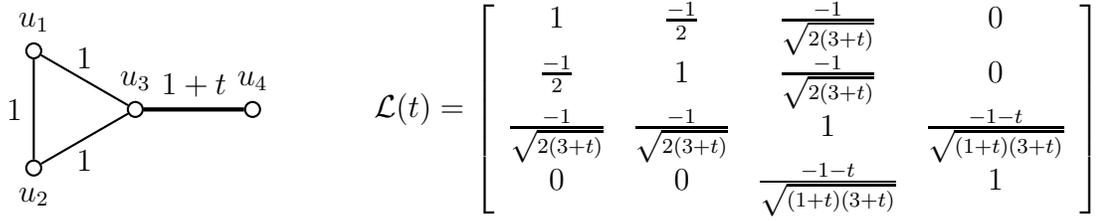

Precisely, the derivative of  a parameterized real-valued symmetric matrix $M(t)$ with simple eigenpair  $(\lambda(t), \mathbf{v}(t))$ where $\mathbf{v}(t)$ is a unit vector is
\begin{align}
    \dlambda{t}(t) = \mathbf{v}^T(t) \frac{dM}{dt}(t) \mathbf{v}(t),
\end{align}\label{general_dlambda}
where  $\frac{dM}{dt}$ is an entry-wise derivative of $M(t)$.

Before moving forward, we note that so far we have discussed that real-symmetric parameterized matrices are differentiable only for simple eigenvalues. In \cite{lancaster64}, it is established that if an eigenvalue is non-degenerate (has a full set of corresponding eigenvectors), then a differentiable neighborhood can still be found. Since real-symmetric matrices are orthogonally diagonalizable, it follows that all eigenvalues we are concerned with are differentiable.

Aksoy et al. in \cite{APY21} found the derivative of eigenvalues of the adjacency, combinatorial Laplacian, and normalized Laplacian matrices. 
We focus on the normalized Laplacian. Here, we use $\v_x$ to mean the $x$th component of $\v$.

\begin{lem}[Aksoy et al. \cite{APY21}]
Let $E_C$ denote a set of edges of graph $G$. 
Let $\lambda$ be a simple eigenvalue of $\L$ of $G$. Then,
\begin{align}\label{the-other-formula}
    \dlambda{E_C} &=(1-\lambda)\sum_{\{x,y\}\in E_C} \left(\frac{\v_{x}^2}{d_x}+ \frac{\v_{y}^2}{d_y} \right)-2\sum_{\{x, y\}\in E_C} \frac{\v_{x} \v_{y}}{\sqrt{d_xd_y}}.
\end{align}
\end{lem}

Since the derivative is dependent on the choice of the eigenvector, for non-simple eigenvalues this is not well-defined. Instead, we can take the derivative over the entire eigenspace, and thus, as show in \cite{APY21} (Lemma 3), the edge derivative will be independent from the particular decomposition of the eigenspace. 

\begin{lem}[Aksoy et al. \cite{APY21}] \label{deriv definition}
Let $E_C$ denote a set of edges of $G$. 
Let $\lambda$ be an eigenvalue of multiplicity $k$ for $\L$ of $G$, and let $V = \{\v_1, \v_2, ... \v_k\}$ be an orthonormal basis for the eigenvectors associated with $\lambda$. Then,
\begin{align}\label{L(G)}
    \dlambda{E_C} &= \frac1k \sum_{i=1}^k\left[
    (1-\lambda)\sum_{\{x,y\}\in E_C} \left(\frac{\v_{i,x}^2}{d_x}+ \frac{\v_{i,y}^2}{d_y} \right)-2\sum_{\{x, y\}\in E_C} \frac{\v_{i,x} \v_{i,y}}{\sqrt{d_xd_y}}
    \right].
\end{align}
\end{lem}

\begin{obs}
This derivative is linear in edges in $E_C$. Therefore, we consider the derivative of an eigenvalue with respect to a single edge knowing we can combine results to find the derivative with respect to any collection of edges.
\end{obs}

The interpretation and results of Lemma \ref{deriv definition} are not constrained to edges. There are meaningful and interesting results found by taking the derivative with respect to a non-adjacent pair of vertices, which we call a \textit{non-edge}. 

Let us begin with a small example of the edge derivative (of both edges and non-edges) and some observations. 
\begin{figure}[ht]
    \centering
    \begin{tikzpicture}
     \node[vertex, label=$v_3$] (v1) at ({0 + cos(0)},{sin(0)}) {};
     \node[vertex, label=$v_1$] (v2) at ({0 + cos(120)},{sin(120)}) {};
     \node[vertex, label=below:$v_2$] (v3) at ({0 + cos(240)},{sin(240)}) {};
     \node[vertex, label=$v_4$] (v4) at ({0 + cos(0) + abs(sin(120) - sin(240))}, {sin(120)}) {};
     \node[vertex, label=$v_5$] (v5) at ({0 + cos(0) + abs(sin(120) - sin(240))}, {sin(240)}) {};
     
     \draw[thick] (v5)--(v1)--(v2)--(v3)--(v1)--(v4);
    \end{tikzpicture}
    \caption{A graph $G$ and $\spec_{\L}(G)=\left\{0, 1, \frac{3}{2}, \frac{5 \pm \sqrt{5}}{4}\right\}$.}
    \label{fig:ex1}
\end{figure}
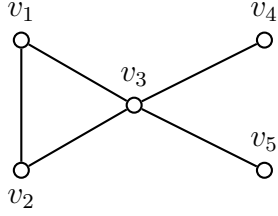

\begin{example}\label{derivatives}

Consider the graph in Figure \ref{fig:ex1} and its corresponding spectrum. Every eigenvalue is simple, and we can find the eigenvectors and calculate the edge derivative for each pair of vertices.
Below are the resulting edge derivatives, separated into edges and non-edges:

\begin{center}
\begin{tabular}{cc}
     \begin{tabular}{|c||r|r|r|r|r|}
    \hline
    \diagbox{$\downarrow e$}{$\lambda \rightarrow$} & $0$ & $1$ & $\frac32$ & $\frac{5 + \sqrt{5}}{4}$ & $\frac{5 -\sqrt{5}}{4}$\\
    \hline
    \hline
    $\{v_1,v_2\}$ & $0$ & $0$ & $\phantom{-}\frac{1}{4}$& $-0.07$ & $-0.18$\\
     \hline
    $\{v_1,v_3\}$ & $0$ & $0$ & $-\frac{1}{8}$& $-0.02$ & $0.15$\\
    \hline
    $\{v_2,v_3\}$ & $0$ & $0$ & $-\frac{1}{8}$& $-0.02$ & $0.15$\\
    \hline
    $\{v_3,v_4\}$ & $0$ & $0$ & $0$& $0.06$ & $-0.06$\\
    \hline
    $\{v_3,v_5\}$ & $0$ & $0$ & $0$& $0.06$ & $-0.06$\\
    \hline
    \end{tabular} &  
     \begin{tabular}{|c||r|r|r|r|r|}
    \hline
    \diagbox{$\downarrow e$}{$\lambda \rightarrow$} & $0$ & $1$ & $\frac32$ & $\frac{5 + \sqrt{5}}{4}$ & $\frac{5 -\sqrt{5}}{4}$\\
    \hline
    \hline
    $\{v_1,v_4\}$ & $0$ & $0$ &$-\frac{1}{8}$& $-0.30$ & $0.43$\\
    \hline
    $\{v_1,v_5\}$ & $0$ & $0$ & $-\frac{1}{8}$& $-0.30$ & $0.43$\\
    \hline
    $\{v_2,v_4\}$ & $0$ & $0$ & $-\frac{1}{8}$& $-0.30$ & $0.43$\\
    \hline
    $\{v_2,v_5\}$ & $0$ & $0$ & $-\frac{1}{8}$& $-0.30$ & $0.43$\\
    \hline
    $\{v_4,v_5\}$ & $0$ & $1$ & $0$ & $-0.72$ & $-0.28$\\
    \hline
    \end{tabular}\\
    (a) Edges & (b) Non-Edges
\end{tabular}
\end{center}

First, we note that the edge derivative of the eigenvalue $\lambda=0$ is always zero. With some knowledge of the eigenvector structure corresponding to this eigenvalue, this is not surprising behavior. 

Observe that for each eigenvalue in the edges table (a), the columns sum to zero. The normalized Laplacian's spectrum is invariant under scaling, meaning changing every edge by $t$ would result in no change to the spectrum. So, this observation is expected. We formally prove that the columns in (a) sum to zero in Theorem \ref{evals-sum-0}. 

In addition, observe that in both tables the rows sum to zero. Since the normalized Laplacian has trace $n$ for any graph, it follows that the sum of the spectrum is constant. We will formally prove this in Theorem \ref{edge-sum-0}. 
\end{example}

The remainder of this section is organized as follows. Section \ref{sec:structure and tools} focuses on the edge derivative for particular graph structures. Section \ref{sec: families} establishes results for general graph families. Finally, we complete our analysis by examining results related to the edge derivative for non-edges in Section \ref{sec: non-edges} and general bounds in Section \ref{sec:bounds}. 
\subsection{Special Graph Structure}\label{sec:structure and tools}

For any graph, the normalized Laplacian always has the eigenvalue $\lambda = 0$, whose multiplicity is the number of connected components in the graph. Since making a small change to an edge would not change the number of connected components in a graph, we expect the edge derivative of this eigenvalue to always be zero. To verify this, there is a well-known eigenvector associated with $\lambda = 0$, namely $D^{\frac12} \mathbbm{1}$.

\begin{thm}\label{lambda-zero}
Let $G$ be a connected graph.
Then, for the eigenvalue $\lambda = 0$ of $\L_G$, the derivative $\dlambda{\{x,y\}}$ is zero with respect to any vertices $x,y$.
\end{thm}

\begin{proof}
Let $G$ be a connected graph. Then, for $\lambda=0$, it follows that $k=1$. Let $x,y$ be any two vertices in $G$.
Therefore,
\begin{align*}
    \dlambda{\{x,y\}} &= \frac{\sqrt{d_x}^2}{d_x}+\frac{\sqrt{d_y}^2}{d_y}-2 \frac{\sqrt{d_x}\sqrt{d_y}}{\sqrt{d_xd_y}} \\
    &= (1 + 1) - 2 \cdot 1 \\
    &= 0. 
\end{align*}
Note that the eigenvector does not need to be normalized because the normalization scalar can be factored out of each sum, and the resulting value will still be 0.
\end{proof}

Since every connected graph has an eigenvalue of 0, and changing the weight of an edge does not change whether a graph is connected, this result is expected.
Additionally, since the eigenvalue derivative is linear in edges, we have the following corollary.
\begin{cor}
Let $E_C$ be a collection of edges of $G$. Then, $\dlambda{E_C}=0$ for the eigenvalue $\lambda = 0$.
\end{cor}

A graph has eigenvalue $\lambda=2$ if and only if the graph is bipartite. Additionally, the multiplicity is the number of bipartite components of the graph. It also has a well-known eigenvector (for a connected graph).  This allows us to establish the following results about the derivatives. 

\begin{thm}\label{lambda-two}
Let $G$ be a connected bipartite graph.
Then, for the eigenvalue $\lambda = 2$ of $\L_G$, the derivative $\dlambda{\{x,y\}}$ is zero with respect to any edge $\{x,y\}$.
\end{thm}

\begin{proof}
Let $G$ be a connected bipartite graph.
The eigenvalue $\lambda = 2$ in $\L$ has multiplicity $k = 1$ and corresponds to the eigenvector $\v=D^\frac12 \begin{bmatrix}
    \mathbbm{1}\\
    -\mathbbm{1}
\end{bmatrix}$.

Since $G$ is bipartite, we know that $x$ and $y$ are in different parts, and so $\v_x$ and $\v_y$ have opposite signs. Thus,
\begin{align*}
     \dlambda{\{x,y\}}&=(1 - 2) \left( \frac{\v_ x^2}{d_x} + \frac{\v_ y^2}{d_y} \right)-2 \frac{\v_x \v_ y}{\sqrt{d_xd_y}}\\
     &= -1 \left(\frac{(\pm \sqrt{d_x})^2}{d_x} + \frac{(\mp \sqrt{d_y})^2}{d_y}\right) - 2 \frac{(\pm \sqrt{d_x}) (\mp \sqrt{d_y})}{\sqrt{d_xd_y}} \\
     &= -(1 + 1) - 2 (-1) \\
     &=0.
\end{align*}
Here as well, the eigenvector does not need to be normalized because the normalization scalar can be factored out of each sum, and the resulting value will still be 0.
\end{proof}

This result is also expected, as changing the weight of an edge does not change whether a graph is bipartite.
That is, if we start with a bipartite graph and change the weight of one edge, the resulting graph will also be bipartite, and thus, will also have 2 as an eigenvalue.

For a bipartite graph, the spectrum of its normalized Laplacian is symmetric about 1.
This means that, for a bipartite graph $G$, if $\lambda \in \spec_{\L}(G)$, then $(2-\lambda)\in \spec_{\L}(G)$.
The eigenvectors of these symmetric eigenvalues also come in pairs, and we show that the edge derivative follows suit. 

\begin{thm} \label{eigval-symmetric}
Let $G=(A,B)$ be a bipartite graph with eigenvalue $\lambda \in \spec_{\L}(G)$. Then \[\frac{d(2-\lambda)}{d\{x,y\}}=-\dlambda{\{x,y\}}\] for any edge $\{x,y\}$. 
\end{thm}

\begin{proof}
Let $\v_1=
\begin{bmatrix}
\u
\\
\w
\end{bmatrix}
$ be an eigenvector for eigenvalue $\lambda_1$ of $\L$ such that $\u$ contains the entries corresponding to the vertices in part $A$ and $\w$ contains the entries corresponding to the vertices in part $B$. 
Since the eigenvalues of the normalized Laplacian are symmetric about 1, it follows that $\v_2=\begin{bmatrix}
    \u
    \\
    \mathbf{-w}
    \end{bmatrix}$ is a unit eigenvector for  the eigenvalue $\lambda_2=2-\lambda_1$ of $\L$. 

Consider the edge derivative of $\lambda_2$ with respect to $\{x,y\}$ where $x$ and $y$ are from different parts of the bipartite graph.
We have
\begin{align*}
    \frac{d\lambda_2}{d\{x,y\}} &= \frac{1}{k} \sum_{i=1}^k \left[ (1-\lambda_2) \left( \frac{\u_{i,x}^2}{d_x} + \frac{(-\w)_{i,y}^2}{d_y} \right) - 2\frac{\u_{i,x}(-\w)_{i,y}}{\sqrt{d_xd_y}} \right]
    \\
    &= \frac{1}{k}\sum_{i=1}^k \left[
    -(1 - \lambda_1) \left(\frac{\u_{i,x}^2}{d_x} + \frac{\w_{i,x}^2}{d_x}\right) + 2\frac{\u_{i,x}\w_{i,y}}{\sqrt{d_xd_y}} \right]
    \\
    &= -\frac{1}{k}\sum_{i=1}^k \left[
    (1 - \lambda_1)\left(\frac{\u_{i,x}^2}{d_x} + \frac{\w_{i,x}^2}{d_x}\right) - 2\frac{\u_{i,x} \w_{i,y}}{\sqrt{d_xd_y}}\right]
    \\
    &= - \frac{d\lambda_1}{d\{x,y\}}. 
\end{align*}
\end{proof}

Another graph structure that relates to the structure of the eigenvectors is twin vertices.
Two vertices $x$ and $y$ are said to be \emph{twin vertices} if their neighborhoods are equal (apart from each other), i.e., if $N(x) \setminus \{y\} = N(y) \setminus \{x\}$.
We say $x$ and $y$ are \emph{connected twins} if they are twin vertices which are adjacent. 
If a graph has twin vertices $x_1,x_2$, then $[1, -1, 0, \ldots, 0]^T$ is an eigenvector for all symmetric graph matrices. 
Using this, we obtain the following results about the edge derivative of the spectrum of a graph with respect to the edge between connected twins.

\begin{thm}\label{twins connected}
Let $G$ be a graph with connected twins $x,y$ such that $|N(x)|=|N(y)|=d$. 
Then, $\lambda= \frac{d+1}{d} \in \spec_{\L}(G)$, and $\dlambda{\{x,y\}} \leq \frac{1}{k} \left( \frac{d-1}{d^2} \right)$, where $k$ is the multiplicity of $\lambda$. 
\end{thm}

\begin{proof}
Let $G$ be a graph with connected twins $x,y$ as described above. Consider a vector $\v$ where $\v_x=\frac{1}{\sqrt{2}}$, $\v_y=-\frac{1}{\sqrt{2}}$, and $\v_i=0$ for all other vertices $i$.  
By computation, $\v$ is an eigenvector for $\L_G$ with eigenvalue $\frac{d+1}{d}$. 

First, suppose the multiplicity $k$ of $\lambda= \frac{d+1}{d}$ is 1 (so $\v$ is the only eigenvector).
Then,
\begin{align*}
    \dlambda{\{x,y\}} & = (1-\lambda) \left(\frac{(1/\sqrt{2})^2}{d}+ \frac{(-1/\sqrt{2})^2}{d}\right) -2 \frac{(1/\sqrt{2})(-1/\sqrt{2})}{d}\\
    &= \left(1-\frac{d+1}{d}\right)\left(\frac{1}{d}\right) + \frac{1}{d}\\
    &= \frac{d-1}{d^2}.
\end{align*}

Now, let $k>1$. 
Since $\v$ is a unit vector, we can extend it to an orthonormal basis of the eigenspace corresponding to $\lambda$, denoted $\{\v, \w_1, \ldots, \w_{k-1}\}$. 
Moreover, for any vector $\w_i$ in this basis, the $x$-th and $y$-th entries will be equal, since $\w_i$ is orthogonal to $\v$, so $\w_{i,x}=\w_{i,y}$ for all $i$. 
Thus, 
\begin{align*}
    \dlambda{\{x,y\}} &= \frac{1}{k} \left[\frac{d-1}{d^2} + \sum_{i=1}^{k-1}\left(1-\frac{d+1}{d}\right)\left(\frac{\w_{i,x}^2}{d}+\frac{\w_{i,y}^2}{d}\right) -2 \frac{\w_{i,x}\w_{i,y}}{d} \right]\\
    &= \frac{1}{k} \left[\frac{d-1}{d^2} + \sum_{i=1}^{k-1}\left(-\frac{d+1}{d}\right)\left(\frac{2\w_{i,x}^2}{d}\right) \right]\\
    & \leq \frac{1}{k} \left( \frac{d-1}{d^2} \right),
\end{align*}
since the degree and $\w_{i,x}^2$ must always be positive. 
\end{proof}

This is maximized when $d=2$ and $k=1$ for a edge derivative of $\dlambda{\{x,y\}}=\frac14$. 
Interestingly, for graphs on a small number of vertices, this is the largest edge derivative value for a single edge.
We explore more extremal values in Section \ref{sec:bounds}.

\begin{cor}\label{cor: twins connected}
Let $G$ be a graph with connected twins $x,y$ such that $|N(x)|=|N(y)|=d$. For all $\lambda \in \spec_{\L}(G)$, if $\lambda \neq \frac{d+1}{d} $, then $\dlambda{\{x,y\}}\leq 0$. 
\end{cor}

\begin{proof}
Consider the set of orthonormal eigenvectors constructed in the proof of Theorem \ref{twins connected}. 
It follows, for all eigenvalues associated with eigenvectors $\w$, that
\begin{align*}
    \dlambda{\{x,y\}} &= \frac{1}{k} \left[ \sum_{i=1}^{k}\left(1-\lambda\right)\left(\frac{\w_{i,x}^2}{d}+\frac{\w_{i,y}^2}{d}\right) -2 \frac{\w_{i,x} \w_{i,y}}{d} \right]\\
    &= \frac{1}{k} \left[\sum_{i=1}^{k}\left(-\lambda\right)\left(\frac{2 \w_{i,x}^2}{d}\right) \right]\\
    & \leq 0. 
\end{align*}
\end{proof}

Looking at all the eigenvalues for a particular edge, we saw in Example \ref{derivatives} that the edge derivatives sum to zero. We will now prove this observation.

\begin{thm}\label{edge-sum-0}
  Let $G$ be a graph with vertices $x$ and $y$. Let $\lambda_1, \ldots, \lambda_n$ be the eigenvalues of the normalized Laplacian of $G$. Then, 
  \begin{align*}
      \sum_{i=1}^n \frac{d \lambda_i}{d\{x,y\}}=0.
  \end{align*}  
\end{thm}

\begin{proof}
    Consider the trace of the normalized Laplacian $\L$ of graph $G$:
    \begin{align*}
        tr(\L) = \lambda_1 + \lambda_2 + \cdots + \lambda_n = n,
    \end{align*}
    since $\L$ has ones along its diagonal. For two vertices $x,y$ in $G$, we can find the edge derivative of both sides to obtain the desired result. 
\end{proof}

We now turn our attention away from particular edges and instead to the edge derivative with respect to \textit{all} edges in the graph.  
 
\begin{thm}\label{evals-sum-0}
Let $\lambda$ be an eigenvalue of the normalized Laplacian $\L$ for a graph $G$. The sum of the derivatives of $\lambda$ over all edges in $G$ is zero.
\end{thm}

\begin{proof}
Suppose the multiplicity of $\lambda$ is 1. Let $E$ be the set of edges in the graph, and let $V$ be the set of vertices. Then the sum of $\dlambda{\{x,y\}}$ over all edges is given by
\begin{align*}
    \sum_{\{x,y\} \in E} \dlambda{\{x,y\}} &= \sum_{\{x,y\} \in E} \left[ (1-\lambda) \left( \frac{\v_x^2}{d_x} + \frac{\v_y^2}{d_y}\right) - 2\frac{\v_x \v_y}{\sqrt{d_xd_y}}\right] \\
    &= (1-\lambda)\sum_{\{x,y\} \in E}  \left( \frac{\v_x^2}{d_x} + \frac{\v_y^2}{d_y}\right) - 2\sum_{\{x,y\} \in E} \frac{\v_x \v_y}{\sqrt{d_xd_y}}.
\end{align*}
In the first sum, $\v_x$ is counted $d_x$ times for each vertex $x$ in the graph. Thus, 
\begin{align*}
    \sum_{\{x,y\} \in E} \dlambda{\{x,y\}} &= (1-\lambda)\sum_{x \in V}  \left( (d_x)\frac{\v_x^2}{d_x}\right) - 2\sum_{\{x,y\} \in E} \frac{\v_x \v_y}{\sqrt{d_xd_y}} \\
    &= (1-\lambda)\sum_{x \in V}  \left( {\v_x^2}\right) - 2\sum_{\{x,y\} \in E} \frac{\v_x \v_y}{\sqrt{d_xd_y}} \\
    &= (1-\lambda)(1) - 2\sum_{\{x,y\} \in E} \frac{\v_x \v_y}{\sqrt{d_xd_y}} 
\end{align*}
because $\v$ is an orthonormal eigenvector, so the squares of its entries sum to 1.

From the derivation of the normalized Laplacian eigenvalue derivative in \cite{APY21}, the remaining sum can be rewritten to get 
\begin{align*}
    \sum_{\{x,y\} \in E} \dlambda{\{x,y\}} &= (1-\lambda)(1) - \left(D^{-\frac12}\v\right)^T A'(t) \left(D^{-\frac12}\v\right),
\end{align*}
where $D$ is the degree matrix of $G$ and $A'(t)$ is the entry-wise derivative of the adjacency matrix evaluated at $t$. Since all edges are parameterized, $A' = A$. 

Therefore,
\begin{align*}
    \sum_{\{x,y\} \in E} \dlambda{\{x,y\}} &= (1-\lambda) - \v^T D^{-\frac12}A D^{-\frac12}\v \\
    &= (1-\lambda) - \v^T (I-\L) \v \\
    &= (1-\lambda) -\v^T\v + \v^T\L\v  \\
    &= (1-\lambda) - (1-\lambda)\v^T\v \\
    &= 0.
\end{align*}
Thus, we get that the sum of the derivatives of $\lambda$ over all edges in $G$ is zero.
\end{proof}

The edge derivative of an eigenvalue $\lambda$ over all edges can be thought of as the effect of changing each edge by the same amount on $\lambda$. Since the normalized Laplacian's spectrum does not change by scaling the graph by a constant, our result is expected. 

\subsection{Families of Graphs}\label{sec: families}

With these tools we can now find the edge derivatives for several families of graphs. 
Let us begin with edge transitive graphs. A graph is \emph{edge transitive} if for every pair of edges $e_1$ and $e_2$, there exists a graph automorphism that maps $e_1$ to $e_2$. Graphs such as the complete graph and cycle are examples of edge transitive graphs. 

\begin{thm}\label{edge-trans}
    Let $G$ be an edge transitive graph. If $\lambda$ is an eigenvalue of $\L_G$, then $\dlambda{e}=0$ with respect to any edge $e$. 
\end{thm}
\begin{proof}
   Consider an eigenvalue $\lambda$ of the normalized Laplacian of edge transitive graph $G$ with edges $e_1, \ldots, e_m$. For any two edges $e_j$ and $e_k$, we know that $\dlambda{e_j}=\dlambda{e_k}$ since there is an automorphism between $e_j$ and $e_k$.
   Let $e$ be an arbitrary edge.
   Therefore,
   \begin{align*}
       \sum_{i=1}^{|E|} \dlambda{e_i}= |E| \dlambda{e}.
   \end{align*}
   From Theorem \ref{evals-sum-0}, we get that the left-hand side is zero. Therefore, $\dlambda{e}=0$.
\end{proof}

Note that edge transitive graphs include complete graphs, cycle graphs, complete bipartite graphs, balanced completed multipartite graphs, and crown graphs.

\subsection{Edge Derivative of Non-Edges}\label{sec: non-edges}
As we have seen previously, the interpretation (and formula) of the edge derivative is not constrained to edges. 
There are also meaningful results found by taking the derivative with respect to non-adjacent pairs of vertices, which we call \emph{non-edges.} 
These derivatives tell us how the eigenvalues change when adding edges to the weighted graph, rather than just when changing the weights of current edges. 
We can also think of a non-edge as an edge of weight zero.
We note that Theorem \ref{lambda-zero} about $\lambda=0$ directly extends to non-edges.

For a bipartite graph, when $\lambda = 2$, the edge derivative now depends on whether the two vertices in the non-edge are from the same part or not (see Figure \ref{fig:bip}). 

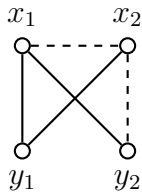
\begin{figure}[ht]
    \centering
    \begin{tikzpicture}[scale=0.7]
    
    
    \pgfmathtruncatemacro{\n}{2}  
    \pgfmathtruncatemacro{\m}{2}  
    
    \pgfmathtruncatemacro{\dyA}{1}  
    \pgfmathtruncatemacro{\dyB}{-1}  
    
    \pgfmathtruncatemacro{\x}{2}  
    
    \foreach \i in {1,...,\n} {
    	\node[vertex, label=$x_\i$] (vA\i) at ({-\x*(\n-1)/2 + \x*(\i-1)}, \dyA) {};
    }
    
    \foreach \j in {1,...,\m} {
    	\node[vertex, label=below:$y_\j$] (vB\j) at ({-\x*(\m-1)/2 + \x*(\j-1)}, \dyB) {};
    }

	\draw[thick] (vA1)--(vB1);
	\draw[thick] (vA2)--(vB1);
	\draw[thick] (vA1)--(vB2);
	
 	\draw[thick,dashed] (vA2)--(vB2); 
 	\draw[thick,dashed] (vA1)--(vA2); 
    
    \end{tikzpicture}
    \caption{A bipartite graph. The pair $\{x_1,x_2\}$ is a non-edge between two vertices in the {same part} and the pair $\{x_2,y_2\}$ is a non-edge between two vertices in {different parts}. Therefore, for $\lambda=2$, $\dlambda{\{x_1,x_2\}}=-\frac{2}{3}$ while $\dlambda{\{x_2,y_2\}}=0$. }
    \label{fig:bip}
\end{figure}

\begin{thm} \label{bipartite-nonedge-diffpart}
Let $G$ be a connected bipartite graph and $\{x,y\}$ be a non-edge between two vertices in different parts of the graph.
Then, for the eigenvalue $\lambda = 2$ of $\L_G$, $\dlambda{\{x,y\}}=0$. 
\end{thm}

\begin{proof}
This follows from the same argument used in the proof of Theorem \ref{lambda-two}.
\end{proof}

\begin{thm}\label{bipartite-nonedge-samepart}
Let $G$ be a connected bipartite graph and $\{x,y\}$ be a non-edge between two vertices from the same part of the graph.
Then, for the eigenvalue $\lambda = 2$ of $\L_G$,  $\dlambda{\{x,y\}}=-\frac{2}{|E|}$.
\end{thm}

\begin{proof}
Let $\{x, y\}$ be a non-edge, where both $x$ and $y$ are in the same part of the bipartite graph. 
The normalized corresponding eigenvector is $$\v=\frac{1}{\sqrt{2 |E|}} D^{\frac12} \begin{bmatrix}
    \mathbbm{1}\\
    -\mathbbm{1}
\end{bmatrix},$$
where $|E|$ is the number of edges in the graph.
Now, since both $x$ and $y$ are in the same part, $\v_x$ and $\v_y$ have the same sign. 
Therefore,
\begin{align*}
    \dlambda{\{x,y\}} &= (1 - 2)\left( \frac{\v_ x^2}{d_x}+  \frac{\v_ y^2}{d_y}\right)-2\frac{\v_x \v_ y}{\sqrt{d_xd_y}} \\
    &= \frac{1}{2|E|} \left[-(1 + 1) - 2 (1)\right] \\
     &= -\frac{2}{|E|}.
\end{align*}
\end{proof}

Looking at the non-edge between isolated twin vertices, we establish similar results to those of connected twins. 
\begin{thm}\label{twins non-adjacent}
Let $G$ be a graph with non-adjacent twins $x$ and $y$ with $|N(x)|=|N(y)|=d$. Then, $\lambda= 1$ is an eigenvalue of $\L_G$, and $\dlambda{\{x,y\}} \leq \frac{1}{kd}$, where $k$ is the multiplicity of $\lambda$. Furthermore, when $k=1$, we have  $\dlambda{\{x,y\}} = \frac{1}{d}$.
\end{thm}

\begin{proof}
Let $G$ be a graph with non-adjacent twins $x$ and $y$ as described above. Consider a vector $\v$ where $\v_x=1/\sqrt{2}$, $\v_y=-1/\sqrt{2}$, and $\v_i=0$ for all other vertices $i$.  
By computation, $\v$ is an eigenvector for $\L_G$ with eigenvalue $\lambda = 1$. 

First, suppose $k=1$.  Then $\v$ is the only eigenvector, and
\begin{align*}
    \dlambda{\{x,y\}} & = (1-\lambda) \left(\frac{\v_x^2}{d}+ \frac{\v_y^2}{d}\right) -2 \frac{\v_x \v_y}{\sqrt{d^2}}\\
    & = 0 + \frac{1}{d}\\
    &= \frac{1}{d}.
\end{align*}

Now, let $k>1$. 
Since $\v$ is a normalized vector, we can extend it to an orthonormal basis of the eigenspace corresponding to $\lambda$.
Denote this basis as $\{\v, \w_1, \ldots, \w_{k-1}\}$. Moreover, for any vector $\w_i$ in this basis, the $x$-th and $y$-th entries are equal since $\w_i$ is orthogonal to $\v$. So $\w_{i,x}=\w_{i,y}$ for all $i$. Thus, 
\begin{align*}
    \dlambda{\{x,y\}} &= \frac{1}{k} \left[\frac{1}{d} + \sum_{i=1}^{k-1}\left(1-1\right)\left(\frac{\w_{i,x}^2}{d}+\frac{\w_{i,y}^2}{d}\right) -2 \frac{\w_{i,x} \w_{i,y}}{d} \right]\\
    &= \frac{1}{k} \left[\frac{1}{d} + \sum_{i=1}^{k-1}(0) - \left(\frac{2\w_{i,x}^2}{d}\right) \right] \\
    &\leq \frac{1}{k} \left( \frac{1}{d} \right)
\end{align*}
since the degree must always be positive. 
\end{proof}

The derivative of this particular eigenvalue resulting from isolated twins is maximized when $d=1$ and $k=1$ for an edge derivative of $\dlambda{\{x,y\}}=1$. For any pair of vertices in graphs on a small number of vertices, this was the largest edge derivative we found.

\begin{cor}\label{cor:twins iso}
Let $G$ be a graph with isolated twins $x,y$ with $|N(x)|=|N(y)|=d$. For all $\lambda \in \spec_{\L}(G)$, if $\lambda \neq 1$, then $\dlambda{\{x,y\}}\leq 0$. 
\end{cor}

\begin{proof}
Consider the set of orthonormal eigenvectors constructed in the proof of Theorem \ref{twins non-adjacent}. It follows for all eigenvalues associated with eigenvectors $\w$, 
\begin{align*}
    \dlambda{\{x,y\}} &= \frac{1}{k} \left[ \sum_{i=1}^{k}\left(1-\lambda\right)\left(\frac{\w_{i,x}^2}{d}+\frac{\w_{i,y}^2}{d}\right) -2 \frac{\w_{i,x} \w_{i,y}}{d} \right]\\
    &= \frac{1}{k} \left[\sum_{i=1}^{k}-\lambda\left(\frac{2 \w_{i,x}^2}{d}\right) \right]\\
    & \leq 0,
\end{align*}
since $\lambda\geq0$.
\end{proof}
The vertices in each part of $K_{m,n}$ are sets of isolated twins. So, the result of Theorem \ref{bipartite-nonedge-samepart} is consistent with Theorem \ref{twins non-adjacent} and Corollary \ref{cor:twins iso}. 

\subsection{Bounds on Edge Derivative}\label{sec:bounds}

Looking generally over all graphs with respect to any edge or non-edge, we establish both an upper and lower bound on the edge derivative. 

\begin{thm} \label{lambda-bounds}
For a connected graph, the derivative of an eigenvalue of the normalized Laplacian $\L$ with respect to an edge $\{x,y\}$ is bounded by
\begin{align*}
    -\lambda \leq \dlambda{\{x,y\}} \leq 2-\lambda.
\end{align*} Moreover, these are tight bounds. 
\end{thm}

\begin{proof}
We know from Formula \ref{L(G)} that the edge derivative is
\begin{align*}
    \dlambda{\{x,y\}} &= \frac1k \sum_{i=1}^k\left[
    (1-\lambda)\left( \frac{\v_{i,x}^2}{d_x} + \frac{\v_{i,y}^2}{d_y}\right)-2 \frac{\v_{i,x} \v_{i,y}}{\sqrt{d_xd_y}}
    \right]\\
     &= \frac1k \sum_{i=1}^k\left[\left(\frac{\v_{i,x}}{\sqrt{d_x}} - \frac{\v_{i,y}}{\sqrt{d_y}}\right)^2 - \lambda\left(\frac{\v_{i,x}^2}{d_x} + \frac{\v_{i,y}^2}{d_y}\right)\right].
\end{align*}
Looking at this rewritten equation for $\dlambda{\{x,y\}}$, it is clear that the first term in the sum, $$\left(\frac{\v_{i,x}}{\sqrt{d_x}} - \frac{\v_{i,y}}{\sqrt{d_y}}\right)^2,$$ is always nonnegative. 
We also know that $\v_{i,x}$ and $\v_{i,y}$ each come from an orthonormal eigenvector, so $$\left(\frac{\v_{i,x}^2}{d_x} + \frac{\v_{i,y}^2}{d_y}\right) \leq 1.$$ 
Thus, if we are trying to minimize $\dlambda{\{x,y\}}$, we get that
\begin{align*}
    \dlambda{\{x,y\}} &= \frac1k \sum_{i=1}^k\left[\left(\frac{\v_{i,x}}{\sqrt{d_x}} - \frac{\v_{i,y}}{\sqrt{d_y}}\right)^2 - \lambda\left(\frac{\v_{i,x}^2}{d_x} + \frac{\v_{i,y}^2}{d_y}\right)\right] \\
    & \geq \frac1k \sum_{i=1}^k \left[0 - \lambda \right] \\
    &= \left(\frac1k \cdot k\right)(-\lambda) \\
    &= -\lambda.
\end{align*}

We can also find an upper bound for $\dlambda{\{x,y\}}$. 
We rewrite $\dlambda{\{x,y\}}$ as 
\begin{align*}
    \dlambda{\{x,y\}} &= \frac1k \sum_{i=1}^k\left[
    (2-\lambda)\left(\frac{\v_{i,x}^2}{d_x} + \frac{\v_{i,y}^2}{d_y}\right) - \left(\frac{\v_{i,x}}{\sqrt{d_x}} + \frac{\v_{i,y}}{\sqrt{d_y}}\right)^2
    \right].
\end{align*}
 To maximize $\dlambda{\{x,y\}}$, because we know that $$\left(\frac{\v_{i,x}}{\sqrt{d_x}} + \frac{\v_{i,y}}{\sqrt{d_y}}\right)^2$$ is always nonnegative, we can underestimate this term as zero. 
Again, we have that $\v_{i,x}$ and $\v_{i,y}$ come from an orthonormal eigenvector, so the sum $$\left(\frac{\v_{i,x}^2}{d_x} + \frac{\v_{i,y}^2}{d_y}\right) \leq 1.$$
Thus,
\begin{align*}
    \dlambda{\{x,y\}} &= \frac1k \sum_{i=1}^k\left[
    (2-\lambda)\left(\frac{\v_{i,x}^2}{d_x} + \frac{\v_{i,y}^2}{d_y}\right) - \left(\frac{\v_{i,x}}{\sqrt{d_x}} + \frac{\v_{i,y}}{\sqrt{d_y}}\right)^2
    \right] \\
    & \leq \frac1k \sum_{i=1}^k\left[(2-\lambda) - 0\right] \\
    &= \left(\frac1k \cdot k\right)(2-\lambda) \\
    &= 2-\lambda.
\end{align*}

Putting these bounds together, we have
\begin{align*}
    -\lambda \leq \dlambda{\{x,y\}} \leq 2-\lambda.
\end{align*}
 Both of these bounds are tight. First, the lower bound is achieved when $\lambda=0$ for any connected graph $G$. 
 The upper bound is achieved when $\lambda=2$ for a bipartite graph. In both these cases, $\dlambda{e}=0$ for any edge $e$ by Theorems \ref{lambda-zero} and \ref{lambda-two}.
\end{proof}

The bounds are also tight when $G$ has a pair of non-isolated twins:
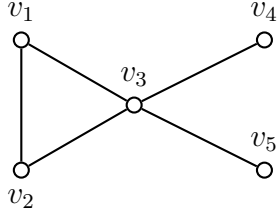
\begin{figure}[ht]
    \centering
    \begin{tikzpicture}
     \node[vertex, label=$v_3$] (v1) at ({0 + cos(0)},{sin(0)}) {};
     \node[vertex, label=$v_1$] (v2) at ({0 + cos(120)},{sin(120)}) {};
     \node[vertex, label=below:$v_2$] (v3) at ({0 + cos(240)},{sin(240)}) {};
     \node[vertex, label=$v_4$] (v4) at ({0 + cos(0) + abs(sin(120) - sin(240))}, {sin(120)}) {};
     \node[vertex, label=$v_5$] (v5) at ({0 + cos(0) + abs(sin(120) - sin(240))}, {sin(240)}) {};
     
     \draw[thick] (v5)--(v1)--(v2)--(v3)--(v1)--(v4);
    \end{tikzpicture}
    \caption{A graph $G$ and $\spec_{\L}(G)=\left\{0, 1, \frac{3}{2}, \frac{5 \pm \sqrt{5}}{4}\right\}$. It has adjacent twins $v_1,v_2$ and non-adjacent twins $v_4,v_5$.}
    \label{fig:fuzzy paw}
\end{figure}

Let us return to the graph from our first example shown again in Figure \ref{fig:fuzzy paw}.
First, observe that for this graph, the normalized Laplacian has eigenvalue $\lambda = 1$ with multiplicity $1$.
For the non-adjacent twin vertices (labeled $v_4$ and $v_5$ in the diagram), notice that both have degree $1$.
By Theorem \ref{twins non-adjacent}, 
$\dlambda{\{v_4,v_5\}} = \frac{1}{d} = 1$. 
This is the upper bound for $\dlambda{\{x,y\}}$ (as shown by Theorem \ref{lambda-bounds})
Therefore, this bound is indeed tight.

For reference, we have included again the edge derivative results which are separated into edges and non-edges:

\begin{center}
\begin{tabular}{cc}
     \begin{tabular}{|c||r|r|r|r|r|}
    \hline
    \diagbox{$\downarrow e$}{$\lambda \rightarrow$} & 0 & 1 & $\frac32$ & $\frac{5 + \sqrt{5}}{4}$ & $\frac{5 -\sqrt{5}}{4}$\\
    \hline
    \hline
    $\{v_1,v_2\}$ & 0 & 0 & $\phantom{-}\frac{1}{4}$& $-0.07$ & $-0.18$\\
     \hline
    $\{v_1,v_3\}$ & 0 & 0 & $-\frac{1}{8}$& $-0.02$ & 0.15\\
    \hline
    $\{v_2,v_3\}$ & 0 & 0 & $-\frac{1}{8}$& $-0.02$ & 0.15\\
    \hline
    $\{v_3,v_4\}$ & 0 & 0 & 0& 0.06 & $-0.06$\\
    \hline
    $\{v_3,v_5\}$ & 0 & 0 & 0& 0.06 & $-0.06$\\
    \hline
    \end{tabular} &  
     \begin{tabular}{|c||r|r|r|r|r|}
    \hline
    \diagbox{$\downarrow e$}{$\lambda \rightarrow$} & 0 & 1 & $\frac32$ & $\frac{5 + \sqrt{5}}{4}$ & $\frac{5 -\sqrt{5}}{4}$\\
    \hline
    \hline
    $\{v_1,v_4\}$ & 0 & 0 &$-\frac{1}{8}$& $-0.30$ & 0.43\\
    \hline
    $\{v_1,v_5\}$ & 0 & 0 & $-\frac{1}{8}$& $-0.30$ & 0.43\\
    \hline
    $\{v_2,v_4\}$ & 0 & 0 & $-\frac{1}{8}$& $-0.30$ & 0.43\\
    \hline
    $\{v_2,v_5\}$ & 0 & 0 & $-\frac{1}{8}$& $-0.30$ & 0.43\\
    \hline
    $\{v_4,v_5\}$ & 0 & 1 & 0& $-0.72$ & $-0.28$\\
    \hline
    \end{tabular}\\
    (a) Edges & (b) Non-Edges
\end{tabular}
\end{center}

Since $v_1$ and $v_2$ are connected twins, for $\lambda=\frac{3}{2}$, we have $\dlambda{\{v_1,v_2\}}=\frac{1}{4}$ by Theorem \ref{twins connected} and Corollary \ref{cor: twins connected}. Furthermore, for $\lambda\neq \frac{3}{2}$, we have $\dlambda{\{v_1,v_2\}}\leq 0$. 
Observe that for each eigenvalue in Table (a), the columns sum to zero (as shown by Theorem \ref{evals-sum-0}); likewise, observe that the rows sum to zero in both tables (as show by Theorem \ref{edge-sum-0}).

Not only is the bound tight for the graph in Figure \ref{fig:fuzzy paw}, but this type of structure gives the largest value of the edge derivative for graphs on small $n$. This leads us to the following conjecture. 

\begin{conj} {\label{conj:max_dir_dir}}
Over all graphs, the maximum edge derivative for any pair of vertices $x,y$ is 1; i.e.
\[\max_{G, \{x,y\}} \dlambda{\{x,y\}}=1.\]
\end{conj}

\section{Edge Derivative of Kemeny's Constant}

Using our formula for the edge derivative of the eigenvalue, we can also obtain a formula for the edge derivative of Kemeny's constant.
This derivative can be used to analyze how the value of Kemeny's constant for a graph changes as one or more of its edge weights are changed. 
Since Kemeny's constant is a measure of the connectivity of a graph, this derivative offers information about how changes in the graph structure affect connectivity.

Recall that Kemeny's constant is calculated in terms of the eigenvalues of $\L$ by
\begin{align*}
    \K(G) = \sum_{\lambda \neq 0} \frac{1}{\lambda}.
\end{align*}
As with the edge derivative of an eigenvalue, $\K$ can be differentiated with respect to a changing edge $\{x, y\}$ in $G$.
Thus, 
\begin{align*}
    \dk{\{x,y\}} = \sum_{\lambda \neq 0} -\frac{1}{\lambda^2}\left(\dlambda{\{x,y\}}\right).
\end{align*}

Here, a positive value of $\dk{\{x,y\}}$ indicates that increasing the weight of edge $\{x, y\}$ causes random walks between vertices to take longer and the graph to be \emph{less} connected. 
On the other hand, a negative value of $\dk{\{x,y\}}$ indicates that increasing the weight of $\{x, y\}$ causes random walks between vertices to get shorter and the graph to become \emph{more} connected.

One might be tempted to assume that adding a non-edge to a graph always results in a negative value of $\dk{\{x,y\}}$ (and increases connectivity). 
This is not the case! 
Take, for example, an ``almost'' barbell graph, i.e., a graph formed by connecting a clique and a clique minus an edge via 
a path (see Figure \ref{fig:almost_barbell}).
Taking the derivative of Kemeny's constant with respect to the non-edge in the clique can in fact yield a positive value. That is, adding this edge will make the graph less connected. 

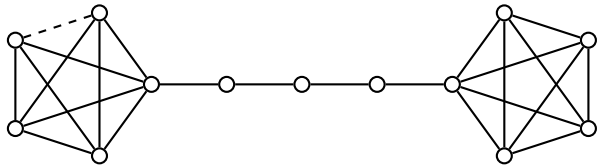
\begin{figure}[ht]
    \centering
                    \begin{tikzpicture}
                        \begin{scope}[xshift=-3cm]
                            \foreach \i in {1,...,5} {
                                \node[vertex] (a\i) at (\i*360/5:1) {};
                             }
                             
                            \foreach \i in {2,...,5} {
                                \foreach \j in {\i,...,5} {
                                    \ifnum\i=\j\relax\else
                                    \draw[thick] (a\i)--(a\j);
                                    \fi
                                }
                            }
                            
                            \foreach \j in {3,...,5} {
                                \draw[thick] (a1)--(a\j);
                            }
                            
                            \draw[thick, dashed] (a1)--(a2);
                            
                        \end{scope}
                        
                        \begin{scope}[xshift=3cm]
                            \foreach \i in {1,...,5} {
                                \node[vertex] (b\i) at (\i*360/5+180:1) {};
                             }
                             
                            \foreach \i in {1,...,5} {
                                \foreach \j in {\i,...,5} {
                                    \ifnum\i=\j\relax\else
                                    \draw[thick] (b\i)--(b\j);
                                    \fi
                                }
                            }
                        \end{scope}
                        
                        \node[vertex] (p1) at (-1,0) {};
                        \node[vertex] (p2) at (0,0) {};
                        \node[vertex] (p3) at (1,0) {};
                        
                        \draw[thick] (a5)--(p1)--(p2)--(p3)--(b5);
                    \end{tikzpicture}
                \
    \caption{An almost barbell becoming a true barbell}
    \label{fig:almost_barbell}
\end{figure}

Similar to the edge derivative of an eigenvalue, we can look at how Kemeny's constant changes if we parameterize every edge. 

\begin{thm} \label{kprime}
The edge derivatives of Kemeny's constant with respect to each edge in a graph $G$ sum to zero.
\end{thm}

\begin{proof}
We begin by writing the sum of $\dk{\{x, y\}}$ over each edge $\{x, y\}$ as
\begin{align*}
    \sum_{\{x,y\} \in E} \dk{\{x,y\}} &=\sum_{\{x,y\} \in E} \left(\sum_{\lambda \neq 0} -\frac{\dlambda{\{x,y\}}}{\lambda^2} \right).
\end{align*}
Switching the order of summation, we find
\begin{align*}
    \sum_{\{x,y\} \in E} \dk{\{x,y\}} &= \sum_{\lambda \neq 0}\left(\sum_{\{x,y\} \in E} -\frac{\dlambda{\{x,y\}}}{\lambda^2}\right) \\
    &= \sum_{\lambda \neq 0} \left[ -\frac{1}{\lambda^2} \left(\sum_{\{x,y\} \in E} \dlambda{\{x,y\}}\right)\right].
\end{align*}
From Lemma \ref{evals-sum-0}, we know that $\sum \dlambda{\{x,y\}} = 0$. Thus,
\begin{align*}
    \sum_{\{x,y\} \in E} \dk{\{x,y\}} &= \sum_{\lambda \neq 0} -\frac{1}{\lambda^2} \cdot 0 \\
    &= 0.
\end{align*}
\end{proof}

\subsection{Families of Graphs}

As with the edge derivative of the eigenvalue, we can find results for the edge derivative of Kemeny's constant for different families of graphs. The proofs for $\dk{\{x,y\}}$ follow nicely and without difficulty from those for $\dlambda{\{x,y\}}$.

\begin{thm}\label{k-prime-complete-biparite-cycle}
Let $G$ be an edge transitive graph. Let $\K$ be Kemeny's constant for $G$ and $\{x, y\}$ be an edge of $G$. Then, $\dk{\{x, y\}} = 0$.
\end{thm}

\begin{proof}
By Theorem \ref{edge-trans}, we have that $\dlambda{\{x, y\}} = 0$ for every eigenvalue $\lambda$ of $G$. Therefore, \[\dk{\{x, y\}} = \sum_{\lambda\ne 0} -\frac{1}{\lambda^2}\left(\dlambda{\{x, y\}}\right) = \sum_{\lambda \ne 0} -\frac{1}{\lambda^2} \cdot 0 = 0,\] as desired.
\end{proof}

\subsection{Bounds and Extreme Values}

Computing the edge derivative of Kemeny's constant with respect to a changing edge allows us to determine exactly which edges in a graph have the greatest impact on overall graph connectivity when their weights are modified.

With this derivative explicitly defined, we can find the bounds on $\dk{\{x,y\}}$ for some edge or non-edge $\{x,y\}$. 
\begin{thm} \label{kbounds}
The edge derivative of Kemeny's constant, $\dk{\{x,y\}}$, is bounded as follows:
\begin{align}
    \sum_{\lambda \neq 0} \left(\frac{-2}{\lambda^2}\right) + \K \leq \dk{\{x,y\}} \leq \K.
\end{align}
\end{thm}
\begin{proof}

We can easily find these bounds by substituting our bounds for $\dlambda{\{x,y\}}$, found in Theorem \ref{lambda-bounds}, into the formula for $\dk{\{x,y\}}$.
Doing this for both the upper and lower bounds of $\dk{\{x,y\}}$, we get that
\begin{align*}
    \sum_{\lambda \neq 0} -\frac{1}{\lambda^2}(2-\lambda) \leq &\dk{\{x,y\}} \leq \sum_{\lambda \neq 0} -\frac{1}{\lambda^2}(-\lambda) \\
    \sum_{\lambda \neq 0} \frac{-2}{\lambda^2} + \sum_{\lambda \neq 0} \frac{\lambda}{\lambda^2} \leq &\dk{\{x,y\}} \leq \sum_{\lambda \neq 0} \frac{\lambda}{\lambda^2} \\
    \sum_{\lambda \neq 0} \frac{-2}{\lambda^2} + \sum_{\lambda \neq 0} \frac{1}{\lambda} \leq &\dk{\{x,y\}} \leq \sum_{\lambda \neq 0} \frac{1}{\lambda}.
\end{align*}

Recall that Kemeny's constant is found by $\K = \sum_{\lambda \neq 0} \frac{1}{\lambda}$. 
Thus, we have the bounds 
\begin{align*}
    \sum_{\lambda \neq 0} \left(\frac{-2}{\lambda^2}\right) + \K \leq \dk{\{x,y\}} \leq \K.
\end{align*}
\end{proof}

Despite the bound for the edge derivative of eigenvalue being tight for certain graphs, it is unknown if either bound for the edge derivative of Kemeny's constant is tight. 
Instead, we present two families of graphs that have the highest and lowest edge derivatives of Kemeny's constant for graphs on up to $7$ vertices and compare these values to the respective bounds.

\subsubsection*{Maximum Value (Lollipop Graphs)}

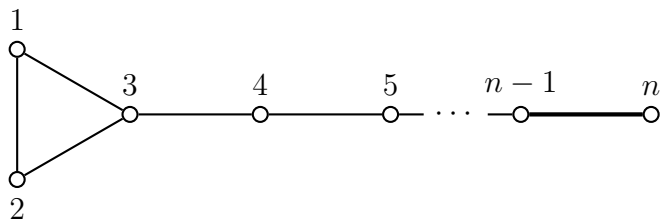
\begin{figure}[ht]
    \centering
        \begin{tikzpicture}[scale=1]
        
            \node[vertex, label=3] (v1) at ({0 + cos(0)},{sin(0)}) {};
            \node[vertex, label=1 ] (v2) at ({0 + cos(120)},{sin(120)}) {};
            \node[vertex, label=below:2] (v3) at ({0 + cos(240)},{sin(240)}) {};
            \node[vertex, label=4] (v4) at ({0 + cos(0) + abs(sin(120) - sin(240))}, 0) {};
            \node[vertex, label=5] (v5) at ({0 + cos(0) + 2 * abs(sin(120) - sin(240))}, 0) {};
            \node[] (dotdotdot) at ({0 + cos(0) + 2.5 * abs(sin(120) - sin(240))}, 0) {$\cdots$};
            \node[vertex, label=$n-1$] (vnminus1) at ({0 + cos(0) + 3 * abs(sin(120) - sin(240))}, 0) {};
            \node[vertex, label=$n$] (vn) at ({0 + cos(0) + 4 * abs(sin(120) - sin(240))}, 0) {};
             
            \draw[thick] (v1)--(v2)--(v3)--(v1);
            \draw[thick] (v1)--(v4)--(v5)--(dotdotdot)--(vnminus1)--(vn);
            \draw[ultra thick] (vnminus1)--(vn);
         
        \end{tikzpicture}
    \caption{The lollipop graph $L_{3,n-3}$ with far edge $\{n-1, n\}$}
    \label{fig:extendo}
\end{figure}

For graphs on $n\leq 7$ vertices, the maximum value of the derivative of Kemeny's constant occurs on the lollipop graph consisting of a complete graph $K_3$ attached by a bridge to a path graph on $n - 3$ vertices (see Figure \ref{fig:extendo}) with respect to the edge at the far end of the path (i.e., the edge with greatest distance from the $K_3$ clique).

\begin{table}[ht]
    \centering
    \begin{tabular}{c|cccccccc}
        $n$ & 5 & 6 & 7 & 8 & 9 & 10 & 11 & 12 \\
        \hline
        Graph & $L_{3, 2}$ & $L_{3, 3}$ & $L_{3, 4}$ & $L_{4, 4}$ & $L_{4, 5}$ & $L_{5, 5}$ & $L_{5, 6}$ & $L_{5, 7}$ \\
        $\dk{\{n-1,n\}}$ & 1.06 & 1.86 & 2.58 & 3.51 & 4.38 & 5.28 & 6.28 & 7.22\\
        Upper Bound & 4.9 & 8.1 & 12.1 & 15.0 & 20.8 & 23.8 & 31.5 & 40.1
    \end{tabular}
    \caption{The maximum edge derivative of $\K$ among lollipop graphs on $n$ vertices with respect to the far edge, as well as the upper bound for that graph as generated by Theorem \ref{kbounds}}
    \label{fig:max-dk}
\end{table}

One interpretation for this graph and edge giving a large value for the edge derivative of $\mathcal{K}$ comes from the characteristics of the barbell graph (see Figure \ref{fig:barbell}). 
It is established that the barbell graph has the highest order of Kemeny's constant at $\mathcal{O}(n^3)$ (see \cite{BBDDLQR}). This high value for $\mathcal{K}$ is understood as a result of the two cliques acting as a sink in a random walk, increasing the mean first passage time.

The high value for the edge derivative of $\mathcal{K}$ on the lollipop graph can be interpreted similarly. 
By increasing the weight of the end of the path, the graph is in a way ``barbellized,'' where the far edge serves the same purpose as a clique in the barbell graph.

In general, the lollipop graph $L_{r, s}$ refers to a complete graph $K_r$ connected to a path graph on $s$ vertices by a bridge.
Table \ref{fig:max-dk} lists the lollipop graph on $n$ vertices (for $5 \leq n \leq 12$) that yields the highest edge derivative of $\K$ with respect to its far edge (labeled $\{n - 1, n\}$), the computed value for this derivative, and the upper bound generated by Theorem \ref{kbounds}. 

The general extremal graph for the upper bound of the edge derivative of Kemeny's constant is likely to be related to the lollipop graph. 
As seen in Table \ref{fig:max-dk}, the value of $\dk{\{n-1,n\}}$ with respect to the far edge of the lollipop grows linearly. Theorem \ref{kbounds} shows that $\dk{\{n-1,n\}}$ is bounded by Kemeny's constant, which itself is bounded by $\mathcal{O}(n^3)$ \cite{ACTT18}.
This is far from the linear growth shown in Table \ref{fig:max-dk}, which leads us to the following conjecture.

\begin{conj} {\label{conj:max_K_dir}}
For a given $n$, the edge derivative of Kemeny's constant on simple graphs of order $n$ is largest on a lollipop graph on $n$ vertices with respect to the far edge. 
The value of this derivative is $\mathcal{O}(n)$.
\end{conj}

\subsubsection*{Minimum Value (Path to Cycle)}

\begin{figure}[ht]
    \centering
    \begin{tikzpicture}[scale=0.7]
            
                \pgfmathtruncatemacro{\dx}{1.4 * abs(sin(120) - sin(240))}
            
                \node[vertex, label=1] (v1) at (0, 0) {};
                \node[vertex, label=2] (v2) at (\dx, 0) {};
                \node[vertex, label=3] (v3) at (2 * \dx, 0) {};
                \node[] (dotdotdot) at (2.5 * \dx, 0) {$\cdots$};
                \node[vertex, label=$n-1$] (vnminus1) at (3 * \dx, 0) {};
                \node[vertex, label=$n$] (vn) at (4 * \dx, 0) {};
                 
                \draw[thick] (v1)--(v2)--(v3)--(dotdotdot)--(vnminus1)--(vn);
                
                \draw[ultra thick,  dashed] (vn) to[out=-135, in=-45] (v1);
             
            \end{tikzpicture}
    \caption{Adding a non-edge to turn a path on $n$ vertices into a cycle}
    \label{fig:path-to-cycle}
\end{figure}
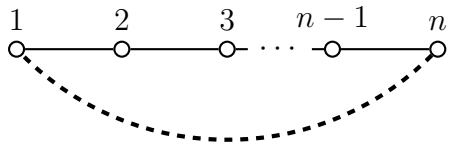

For graphs on $n\leq 7$ vertices, the minimum value of the edge derivative of Kemeny's constant occurs on the path graph with respect to the \textit{non-edge} that forms a cycle graph when added (see Figure \ref{fig:path-to-cycle}). 
Especially for larger values of $n$, adding this edge has a dramatic effect on Kemeny's constant, as shown in Table \ref{fig:min-dk}. Compared to the maximal value, this derivative's magnitude is much larger.

\begin{table}[ht]
    \centering
    \begin{tabular}{c|cccccccc}
        $n$ & 5 & 6 & 7 & 8 & 9 & 10 & 11 & 12\\
        \hline
        $\dk{\{n,1\}}$ & $-9.4$ & $-19.2$ & $-34.0$ & $-54.9$ & $-82.7$ & $-118.5$ & $-163.4$ & $-218.2$\\
        Lower Bound & $-21$ & $-52$ &$-111.2$ & $-208$ & $-357$ & $-574.2$ & $-877.8$ & $-1288$\\
    \end{tabular}
    \caption{The edge derivative of $\K$ with respect to the path-to-cycle non-edge on $n$ vertices, as well as the lower bound for that graph as generated by Theorem \ref{kbounds}}
    \label{fig:min-dk}
\end{table}

Again, this behavior can be described by considering the connectivity of the changing edge. 
The cycle graph $C_n$ is much more connected than the path graph $P_n$, since the maximum distance between any two points in a cycle $C_n$ is $\lfloor n/2 \rfloor$, while the maximum distance between any two points in a path $P_n$ is $n - 1$.
We would then expect a random walk between two arbitrary vertices to be much shorter on average in $C_n$ than in $P_n$.
The computed values of the edge derivative confirm this expectation.

\begin{table}[ht]
    \centering
    \begin{tabular}{c|cccccccc}
        $e$ & $\{1,8\}$ & $\{1,9\}$ & $\{1,10\}$ & $\{1,11\}$ & $\{1,12\}$ & $\{1,13\}$ & $\{1,14\}$ & $\{1,15\}$\\
        \hline
        $\dk{e}$ & $-139.4$ & $-191.8$ & $-248.6$ & $-306.4$ & $-361.1$ & $-407.8$ & $-440.6$ & $-452.7$\\
        
        $\Delta \K$ & $-16.3$ & $-20.4$ &$-24.0$ & $-27.0$ & $-29.1$ & $-30.1$ & $-29.9$ & $-28.2$\\
    \end{tabular}
    \caption{The edge derivative of $\K$ with of a specified non-edge for the path on $15$ vertices, as well as the difference of Kemeny's constant for the two graphs.}
    \label{fig:eight-ninths}
\end{table}

Interestingly, our results differ from the greatest decrease of Kemeny's constant for a single edge change. The greatest decrease appears when adding an edge from a path end to about $\frac 89$ along the path creating a cycle with a pendant path (see \cite{KZ16}). As shown in Table \ref{fig:eight-ninths}, we found the smallest edge derivative for Kemeny's constant to occur when adding an edge between two path ends.

For small $n$, this ``path to cycle'' example is the extremal example. The extremal example for larger $n$ is likely related. 
As seen in Table \ref{fig:min-dk}, the value for $\dk{\{n,1\}}$ on the path grows far slower than the growth of the bound from Theorem \ref{kbounds}, which is $\mathcal{O}(n^6)$ (since the eigenvalues are $\mathcal{O}(n^3)$). This leads us to to the following conjecture.

\begin{conj} {\label{conj:min_K_dir}}
For a given $n$, the derivative of Kemeny's constant on simple graphs of order $n$ is smallest on the path on $n$ vertices with respect to the non-edge connecting the ends of the path. The value of this derivative is $\mathcal{O}(n^3)$.
\end{conj}

\section{Future Directions}
In this paper, we provided bounds for the edge derivative of eigenvalues of the normalized Laplacian (Theorem \ref{lambda-bounds}) and for the edge derivative of Kemeny's constant (Theorem \ref{kbounds}). However, as the empirical data and our Conjectures \ref{conj:max_dir_dir}, \ref{conj:max_K_dir}, and \ref{conj:min_K_dir} suggest, these bounds can probably be improved. In particular, the current bounds give Kemeny's constant bounded above by $\mathcal{O}(n^3)$ and below by $\mathcal{O}(n^6)$, but empirical data suggest that more appropriate bounds are above by $\mathcal{O}(n)$ and below by $\mathcal{O}(n^3)$.

Additionally, our focus was on unweighted, undirected graphs. Future work includes extending results to weighted derivatives. Instead of parameterizing the graph by $t$, each parameterized edge $\{x,y\}$ would have value $w_{x,y}(1+t)$. The equation for the edge derivative of an eigenvalue of the normalized Laplacian would then become 
\begin{align*}
    \dlambda{E_C} &= \frac1k \sum_{i=1}^k\left[
    (1-\lambda)\sum_{\{x,y\}\in E_C} w_{x,y}\left(\frac{\v_{i,x}^2}{d_x}+ \frac{\v_{i,y}^2}{d_y} \right)-2\sum_{\{x, y\}\in E_C} \frac{w_{x,y} \v_{i,x}\v_{i,y}}{\sqrt{d_xd_y}}
    \right].
\end{align*}
If all the parameterizing weights are the same, then this has no effect. 
It is straightforward to compute that Theorems $\ref{lambda-zero}$ and $\ref{lambda-two}$ stay consistent in this extension.  

Extending results to directed graphs is more challenging. 
This is because the definition of the derivative of a matrix-eigenvector equation is only defined for Hermitian matrices.
In the case of real-valued graph matrices, this constrains us to symmetric matrices. 
The normalized Laplacian has many spectral similarities to the probability transition matrix $D^{-1}A$, which is not symmetric.
If $\L\v=\lambda \v$ and $\u =D^{-1/2}\v$, then $D^{-1}A\u=(1-\lambda)\u$. 
For a connected graph, the edge derivative of the eigenvalues, $\mu$, of the probability transition matrix would be as follows,  
\begin{align*}
    \frac{d \mu}{E_C} &= \frac1k \sum_{i=1}^k\left[
    \mu\sum_{\{x,y\}\in E_C} (u_{i,x}^2 + u_{i,y}^2)-2\sum_{\{x,y\}\in E_C} u_{i,x}u_{i,y}
    \right].
\end{align*} This could be modified to include accommodate directed graphs as follows, 
\begin{align*}
    \frac{d \mu}{E_C} &= \frac1k \sum_{i=1}^k\left[
    \mu\sum_{(x,y)\in E_C} u_{i,x}^2 -\sum_{(x,y)\in E_C} u_{i,x}u_{i,y}
    \right].
\end{align*} These results would further expand our knowledge of how small changes to directed and weighted graphs affect Kemeny's constant. 

\section*{Acknowledgement}
This research was conducted primarily at the 2022 Iowa State University Math REU which
was supported through NSF Grant DMS-1950583.

\end{document}